\documentclass[11pt]{amsart}
\usepackage{amsmath,amsfonts,amssymb,amsthm,epsfig,verbatim,ifthen,graphicx}
\usepackage[usenames,dvips]{color}
\usepackage[all]{xy}
 \numberwithin{equation}{section}

\theoremstyle{definition} 
\usepackage{amsmath}
\usepackage{amsfonts}
\usepackage{amstext}
\usepackage{amsbsy}
\usepackage{amsopn}
\usepackage{amsxtra}
\usepackage{upref}
\usepackage{amsthm}
\usepackage{amsmath}
\usepackage{amssymb}
\usepackage{tikz}
\usetikzlibrary{automata,positioning}
\usepackage{enumitem}
\usepackage{epsfig,verbatim,ifthen,graphicx}
\usepackage[all]{xy}
\usepackage{epsfig,verbatim,ifthen,graphicx}

\newtheorem{prop}{Proposition}[section]
\newtheorem{lemma}{Lemma}[section]
\newtheorem{thm}{Theorem}[section]
\newtheorem{ex}{Example}[section]
\newtheorem{coro}{Corollary}[section]

\theoremstyle{definition}
\newtheorem{defi}{Definition}[section]
\newtheorem{rem}{Remark}[section]
\usepackage{enumitem}

\def\R{{\mathbb R}}
\def\C{\ensuremath{\mathbb C}}
\def\N{{\mathbb N}}

\def\F{{\mathcal F}}

\def\G{{\mathcal G}}

\begin{document}
\title
{Existence of Gibbs measures for sequences of continuous functions  } 


\author{Yuki Yayama}
\address{Centro de Ciencias Exactas, 
Departamento de Ciencias B\'{a}sicas,  Universidad del B\'{i}o-B\'{i}o, Avenida Andr\'{e}s Bello 720, Casilla 447, Chill\'{a}n, Chile}
\email{yyayama@ubiobio.cl}

\begin{abstract}
We  give a necessary and sufficient condition for the existence of
 invariant Gibbs measures for sequences of continuous functions on one-sided subshifts
 and, more generally, for the existence of Gibbs measures. These extend the results of  Kim \cite{kim} and Baker and Ghenciu \cite{BG}, respectively,  to sequences of continuous functions on one-sided subshifts.
In particular, we characterize the existence of 
invariant Gibbs measures for subadditive sequences.
For superadditive sequences on subshifts with the strong specification property, 
such a characterization gives a necessary and sufficient condition for the uniqueness of 
invariant Gibbs measures. 
We apply these results to study some problems in the theory of relative pressure and relative equilibrium states.  

\end{abstract}
\maketitle
\section{Introduction}\label{intro}
The non-additive thermodynamic formalism is a generalization of the formalism for continuous functions, 
where we consider  sequences of continuous functions instead of continuous functions.  
In the thermodynamic formalism for continuous functions, existence and uniqueness of invariant Gibbs measures 
have been intensively
studied in relation to the study of equilibrium states.
In particular, it is known that if a continuous function $f$ belongs to the Bowen class on a one-sided subshift with the weak 
specification property, then there exists a unique invariant Gibbs measure for $f$ \cite{Bo}. 
In the non-additive 
thermodynamic formalism,
the notion of  Gibbs measures is generalized for sequences of continuous functions.  
Barreira \cite{b2} and Mummert \cite{m} proved that for an almost additive sequence with 
bounded variation on a one-sided subshift with the strong specification property, there exists a unique invariant Gibbs measure.
More generally, Feng \cite{Fe} proved that a sequence  (a so-called quasi-multiplicative sequence) of continuous functions  has a unique invariant Gibbs measure.  The quasi-multiplicativity implies the weak specification property 
of a one-sided subshift. These sequences are applied to study dimension theory under non-conformal settings and the 
variational principle holds.  
Quasi-multiplicative sequences also arise naturally in the thermodynamic formalism for 
matrix cocycles (see for example \cite{P}). 
Recently, 
the existence and uniqueness of invariant Gibbs measures for superadditive sequences $t\F$, where $\F$ is a quasi-multiplicative sequence associated with a matrix cocyle and $t<0$, were proved under certain settings and certain conditions on $t$
(see \cite{R, MV, MQ}).
   
Although the existence of invariant Gibbs measures for sequences of continuous functions has  
been studied widely, a necessary and sufficient condition for the existence has not been studied. In this paper we study 
a condition for the existence of invariant Gibbs measures for 
sequences of continuous functions on arbitrary one-sided subshifts.

Recently, Kim \cite{kim} gave a necessary and sufficient condition for 
the existence of invariant Gibbs measures for continuous functions on two-sided subshifts over finitely many symbols, studying the property called ``balanced subshift''. 
The notion of ``balanced subshift'' for two-sided subshifts was first introduced by Baker and Ghenciu \cite{BG}  with respect to the continuous function $f=0$ and it was applied to give a necessary and sufficient  condition for 
the existence of Gibbs measures for continuous function $f=0$ \cite[Theorem 3.14]{BG}.
We first extend the notion to general sequences of continuous functions on one-sided subshifts
 by introducing the following  slightly more general form of the ``balanced subshift" condition. 

Let $(X,\sigma_X)$ be a one-sided subshift over finitely many symbols and 
$\F=\{\log f_n\}_{n\in \N}$ be a sequence of continuous functions on $X$. 
Denote by $B_n(X)$ the set of allowable words of length $n\in \N$ in $X$. Given an allowable word $u$ in $X$, for each $n\in \N$, define  $F_n(u):=\{v\in B_n(X): uv \text{ is an allowable word in } $X$\}$ and $P_n(u) =\{v\in B_n(X): vu \text{ is an allowable word in } $X$\}$.
\begin{defi}\label{Balance}
A subshift $X$ is {\em right balanced with respect to $\F$}  if there exist $C\geq 1$ and $K\in \N$ such that for each $m\in\N, n\geq K$ and $u\in  B_m(X)$
\begin{equation}\label{RB}
\frac{1}{C}\leq\frac{\sum_{v\in F_n(u)}f_{m+n}(x_{uv})}{f_m(x_u)\sum_{w\in B_{n}(X)}f_{n}(x_{w})}\leq C
\end{equation}
for each choice of $x_{u}\in[u]$, $x_{uv}\in [uv]$, $x_{w}\in [w]$. 
A subshift $X$ is {\em left balanced with respect to $\F$}  if there exist $C\geq 1$ and $K\in \N$ such that for each $m\in\N, n\geq K$ and $u\in  B_m(X)$
\begin{equation}\label{LB}
\frac{1}{C}\leq\frac{\sum_{v\in P_n(u)}f_{m+n}(x_{vu})}{f_m(x_u)\sum_{w\in B_{n}(X)}f_{n}(x_{w})}\leq C
\end{equation}
for each choice of $x_{u}\in[u]$, $x_{vu}\in [vu]$, $x_{w}\in [w]$. A subshift $X$ is {\em balanced with 
respect to $\F$}  if it is right balanced and left balanced with respect to $\F$. It is 
{\em one-sided balanced with respect to $\F$} if it is left or right 
balanced with respect to $\F$.
\end{defi}
In the original definitions introduced by \cite{kim}, $(X, \sigma_X)$ is a two-sided subshift, $f_n=e^{S_nf}$ where $f\in C(X)$, and 
$K=1$. Under this setting, Kim \cite[Theorem 3.14]{kim} proved that  there exists an 
 invariant Gibbs measure for 
$f\in C(X)$ if and only if $X$ is balanced with respect to $f$.

In Theorem \ref{main} \ref{m1}, we first give a necessary and sufficient  condition for the existence of 
Gibbs measures for sequences of continuous functions on  one-sided subshifts $X$ by using the
right balanced property. This extends the result of Baker and Ghenciu
\cite[Theorem 3.14]{BG}. 
In Theorem \ref{main}\ref{m2}, we extend  \cite[Theorem 3.14]{kim} by showing that there is an 
invariant Gibbs measure for 
a sequence $\F$ on $X$  if and only if $X$ is balanced with respect to $\F$. 
Here $\F$ is an arbitrary sequence of continuous functions and $X$ is an arbitrary one-sided subshift. 
Our proof is direct: we construct a Gibbs measure by using the right balanced property 
and then using the Gibbs measure and the left balanced property we construct an invariant Gibbs measure.  
Using the proof, in Corollary \ref{onesidedGibbs}, we also study a condition for a sequence $\F$ on $X$ to have an invariant measure $\mu$ satisfying only
the lower bound or upper bound  property of Gibbs measures.


The balanced condition on a subshift with respect 
to a sequence of continuous functions does not imply the ergodicity of an invariant Gibbs measure nor the uniqueness of Gibbs measures (Example \ref{ex1}). 
In section \ref{onlysub}, we first characterize the existence of invariant Gibbs measures for 
subadditive sequences of continuous functions  
(Theorem \ref{mishiro}). For superadditive sequences on subshifts with the strong specification property, 
the balanced condition is a necessary and sufficient condition for the uniqueness of 
invariant Gibbs measures. The variational principal for 
the topological pressure holds and the unique invariant Gibbs measure is an ergodic 
equilibrium state for $\F$ (Theorem \ref{mishiroU}).
The property of quasi-multiplicative sequences is studied using balanced subshifts (Propositions \ref{submulti} and 
\ref{ncl})
and examples of subshifts satisfying the balanced property are given. 

In section \ref{apli}, we apply Theorem \ref{main}\ref{m2} to study some problems in the theory of relative pressure. Given a one-block factor maps $\pi: X \to Y $ between one-sided subshifts with the strong specification property and a unique equilibrium 
state $\mu$ for a subadditive sequence $\G$ on $Y$ which is also Gibbs for $\G$, we consider 
preimage measures $\nu$ on $X$ such that $\nu$ is a measure of maximal relative entropy over $\mu$.
Such a measure $\nu$ is an equilibrium state for a sequence of continuous functions which is in general neither subadditive nor superadditve.
Yoo \cite{Yo} studied factor maps between two-sided shift spaces and showed the uniqueness of such $\nu$ when  $\mu$ is Gibbs for a H\"{o}lder continuous function on an irreducible shift of finite type $X$. However, properties of $\nu$ were not studied.
More generally, if $\mu$ is a unique equilibrium 
state for a certain subadditive sequence $\G$ on $Y$ which is also Gibbs for $\G$, such a measure $\nu$ appears as a measure of full Hausdorff dimension of a compact invariant set 
of a certain expanding map \cite{Fe, Y2}. However, the Gibbs property of such a measure has not been fully studied. 
In Theorem \ref{general}, we 
study an equivalent condition for such measures to be Gibbs.   
 Finally in Proposition \ref{yy} we study factors of invariant Gibbs measures in terms of the balanced property.

\section{Background}

\subsection{Sequences of continuous functions.}\label{seqmany}
 We say that $(X, \sigma_X)$ is a {\em one-sided subshift} over $\{1,\dots, k\}$ if $X$ is a closed
shift-invariant subset of $\Sigma_{k}^{+}:=\{1,\dots, k\}^{\N}$ for some $k\geq
1$, i.e., $\sigma_X(X)\subseteq X$,  where the shift 
$\sigma_X:X\rightarrow X$
is defined by $(\sigma_X(x))_{i}=x_{i+1}$ for all $i\in \N$, $x=(x_n)^{\infty}_{n=1} \in X.$
Throughout the paper, we consider one-sided subshifts. 
Define a metric $d$ on $X$ by $d(x,x')={1}/{2^{k}}$ if
$x_i=x'_i$ for all $1\leq i\leq k$ and $x_{k+1}\neq {x'}_{k+1}$, $d(x,x')=1$ if $x_1\neq x'_1$,
and $d(x,x')=0$ otherwise. Define $B_{0}(X)=\{\epsilon\},$ where  $\epsilon$ is the empty word of length $0$.
The language of $X$ is the set $B(X)=\cup_{n=0}^{\infty}B_n(X)$.  
Given  $u\in B_n(X)$, define a cylinder set $[u]$ of length $n$ in $X$  by 
$[u]=\{(x_i)_{i=1}^{\infty} \in X:  x_1\cdots x_n= u\}.$  
A subshift $(X,\sigma_X)$ is {\it irreducible} if for any allowable words $u, v\in B(X)$, there exists  $w\in B(X)$ such that $uwv \in B(X)$, and has the {\it weak specification property} if there exists $k\in\N$ such that for 
any allowable words $u, v \in B(X)$, there exist  $0\leq i\leq k$ and $w\in B_{i}(X)$ such that $uwv\in B(X)$. A subshift has the {\it strong specification property}
if we have $w\in B_{k}(X)$ for each $u, v\in B(X)$. We call such a $k$ a specification number.
 
Given a subshift $(X, \sigma_X)$,  for each $n\in\N$, let $f_n: X\rightarrow \R^{+}$ be a continuous function. Then $\F=\{\log f_n\}_{n=1}^{\infty}$ is a sequence of continuous functions on $X$.
A sequence $\F$ is {\em subadditive} on $X$ if there exists a constant $C\geq 0$ 
such that 
\begin{equation}\label{aa1}
f_{m+n}(x) \leq e^{C}f_m(x) f_{n}(\sigma^m_X x) 
\end{equation}
for every $x\in X$, $n,m\in\N$, and 
it is {\em superadditive} on $X$ if there exists a constant $C\geq 0$ 
such that 
\begin{equation}\label{aa2}
f_{m+n}(x) \geq e^{-C}f_m(x) f_{n}(\sigma^m_X x) 
\end{equation}
for every $x\in X$, $n,m\in\N$. In the theory of the non-additive thermodynamic formalism 
(see for example \cite{b2, m, CFH}), the definitions of the Bowen class, topological pressure and Gibbs measure for continuous 
functions are generalized to those for sequences of continuous  functions.
A sequence $\mathcal{F}= \{ \log f_n \}_{n=1}^{\infty}$ has {\it bounded variation}  if there exists $M 
\geq 1$ such that
 $\sup \{ M_n : n \in \N\} \leq M$ where
 $M_n= \sup \{ {f_n(x)}/{f_n(y)} : x,y  \in X, x_i=y_i \textrm{ for }$
 $1 \leq i \leq n\}.$ 
 For a sequence $\F$ and $n\in\N$, define $Z_n(\F):=\sum_{u\in B_n(X)}\sup\{f_n(x): x\in [u]\}$. 
 We define a topological pressure $P(\F)$ for $\F$ by 
 $P(\F):=\limsup_{n \to \infty}({1}/{n})\log Z_n(\F)$.  Let $M(X, \sigma_X)$ be the set of  invariant Borel probability measures on $X$.
A measure $\mu\in M(X, \sigma_X)$ is an {\em equilibrium state} for $\F$ if 
\begin{equation*}
h_{\mu}(\sigma_X)+\lim_{n\rightarrow\infty}\frac{1}{n}\int\log f_nd\mu
=\sup_{m\in M(X,\sigma_X)}\{h_{m}(\sigma_X)+\lim_{n\rightarrow\infty}\frac{1}{n}\int\log f_ndm\}.
\end{equation*} 
The variational principal holds for subadditive sequences. 
 A Borel probability measure $\mu$ on $X$ is a {\em Gibbs measure} for 
a sequence $\F$ if there exist $P\in \R$ and $C> 0$  such that for each $u\in B_n(X), n\in \N$
\begin{equation}\label{gibbsd}
\frac{1}{C}\leq \frac{\mu([u])}{e^{-nP}f_n(x)}\leq C
\end{equation}
for every $x\in [u]$ and $n\in \N$. More generally, if $\mu([u])/({e^{-nP}f_n(x)})\leq C$, then we say that $\mu$ 
satisfies the upper bound property of 
(\ref{gibbsd}).

It is known by \cite{Fe} that there exists a unique invariant Gibbs measure for a sequence $\F=\{\log f_n\}_{n\in \N}$ with bounded variation 
on a subshift $X$
satisfying the following properties \ref{a0} and \ref{a1}.  
\begin{enumerate}[label=(C\arabic*)]
\item \label{a0} The sequence $\F$ is subadditive. 
\item There exist $k\in\N$ and $0<D<1$ such that  given any $u \in B_m(X), v \in B_n(X)$, $n, m\in\N$, 
there exists $w \in B_i(X), 0\leq i\leq k$ such that  $uwv\in B(X)$ and \label {a1} 
\begin{equation}\label{original}
\sup\{f_{m+n+i}(x): x\in [uwv]\} \geq D \sup\{f_m(x) : x\in [u]\}
\sup\{f_n(x): x\in [v]\}.
\end{equation}
\end{enumerate}
The unique invariant Gibbs measure is a unique equilibrium state for $\F$.
The property of bounded variation  implies that  (\ref{original}) in \ref{a1} is equivalent to:  there exists $0<D'<1$ such that  
\begin{equation}\label{alter}
f_{m+n+i}(x_{uwv}) \geq D' f_m(x_{u}) f_n(x_{v})
\end{equation}
for each choice of $x_{uwv}\in [uwv], x_u\in [u], x_v\in [v]$ 

\begin{rem}
 The sequences with bounded variation satisfying \ref{a0} and \ref{a1} are equivalent to the sequences called 
 quasi-multiplicative sequences of continuous functions (see \cite{Fe}). 
 By setting  
 $\phi_n(x)=\sup\{f_n(x): x\in [u]\}$ we obtain a quasi-multiplicative sequence. The reverse implication is obvious.
\end{rem}
 

\subsection{Balanced properties of a subshift with respect to a sequence of continuous functions} 
The notion of the balanced property and boundedly supermultiplicative property of subshifts was first introduced by Baker and Ghenciu \cite{BG} with respect to the continuous function $f=0$. It was generalized to subshifts with respect to   
arbitrary continuous functions in \cite{kim}.
We  generalize the supermultiplicative property to sequences of continuous functions.
Let $(X, \sigma_X)$ be a one-sided subshift and $\F=\{\log f_n\}_{n\in\N}$ be a sequence of continuous functions on $X$.  
\begin{defi}\label{dBSM}
A subshift $X$ is {\em boundedly supermultiplicative with respect to $\F$ (BSM($\F$))} if there exist $C\geq 1$ and
$K\in \N $
such that for each $m\in\N, n\geq K$  
\begin{equation}\label{bms}
\frac{1}{C}\leq\frac{\sum_{u\in B_m(X)}f_{m}(x_{u})\sum_{v\in B_n(X)}f_n(x_{v})}{\sum_{w\in B_{m+n}(X)}f_{m+n}(x_{w})}\leq C
\end{equation}
for each choice of $x_{u}\in[u]$, $x_{v}\in [v]$, $x_{w}\in [w]$. 
\end{defi}
\begin{rem}\label{kimcase}
The balanced properties and boundedly supermultiplicative property in \cite{BG, kim} are defined for two-sided subshifts. 
To obtain the definitions in \cite{kim}, 
set $K=1$ and let $f_n=e^{S_nf}$ for $f\in C(X)$, where $(S_nf)(x):=\sum_{i=0}^{n-1}f(\sigma_X^{i}x)$ for each $x\in X$, in Definitions \ref{Balance} and  \ref{dBSM}. 
\end{rem} 

If $X$ is one-sided balanced with respect to $\F$, then $\F$ has bounded variation.
To see this, assume that it is right balanced to $\F$ and let $u\in B_m(X)$ and $x, y\in [u]$. Now
replace $x_u$ by $x$ in (\ref{RB}) and then replace $x_u$ by $y$ in (\ref{RB}). Now we first show that the balanced property of a subshift with respect to a sequence of continuous functions is 
a generalization of conditions \ref{a0} and \ref{a1} of a sequence with bounded variation.
 
 \begin{prop}\label{submulti}
Let $\F=\{\log f_n\}_{n\in\N}$ be a sequence of continuous functions with bounded variation on a subshift $X$ 
satisfying \ref{a0} and \ref{a1}.
Then $X$ is balanced with respect to $\F$.
\end{prop}
\begin{proof}
Since $\F$ has bounded variation, there exists $M\geq 1$ such that  $f_n(x)/f_n(y)\leq M$ for every $x, y\in [u],
u \in B_n(X), n\in \N$. 
Let $k$ be defined from \ref{a1}. We first show that there exists $C_1>1$ such that for every $m\in\N$, 
$n\geq k+1$ and $u\in  B_m(X)$
\begin{equation}
\frac{1}{C_1}\leq\frac{\sum_{v\in F_n(u)}f_{m+n}(x_{uv})}{f_m(x_u)\sum_{w\in B_{n}(X)}f_{n}(x_{w})}
\end{equation}
for each choice of $x_{u}\in[u]$, $x_{uv}\in [uv]$, $x_{w}\in [w]$. 
Since $\F$ has bounded variation, 
by \ref{a1} , there exist $0<D'<1$ and $k\in \N$ such that 
given $u\in B_m(X)$ and $v\in B_{n-k}(X)$, $n\geq k+1$, there exists $w\in B_p(X)$, $0\leq p\leq k$, such that 
$f_{m+n-k+p}(x_{uwv})\geq D'f_m(x_u)f_{n-k}(x_v)$, for each choice of $x_{uwv}\in [uwv], x_u\in [u], x_v\in [v]$. 
Take such a $w$. If $p<k$, then take $v'\in B_k(X)$. Then given $uwv$, where $\vert w\vert <k$, and $v' \in B_k(X)$, there exists $w'\in B_{p'}(X)$,
$0\leq p'\leq k$ such that 
\begin{equation}
\begin{split}
f_{m+n+p +p'}(x_{uwvw'v'})&\geq D' f_{m+n+p-k}(x_{uwv})f_k(x_{v'})\\& \geq {D'}^2f_{m}(x_{u})f_{n-k}(x_v)f_k(x_v')\\
\end{split}
\end{equation}
for each choice of $x_{uwvw'v'}\in [uwvw'v']$, $x_{uwv}\in [uwv]$, $x_{v'}\in [v'], x_u\in [u], x_v\in [v]$.  
Now write $w'=w_1'\dots w'_{p'}$ and $v'=v_1'\dots v_k'$ and define $v'':=[w'v']_{1}^{k-p}$ and $v''':=[w'v_1]_{k-p+1}^{p'+k}. $
Then by subadditivity and bounded variation, 
\begin{equation} \label{eqk0}
\begin{split}
f_{m+n}(x_{uwvv''})f_{p+p'}(x_{v'''})e^CM^2&\geq f_{m+n+p'+p}(x_{uwvw'v'})\\
&\geq {D'}^2 f_m(x_u)f_{n-k}(x_v)f_k(x_{v'})
\end{split}
\end{equation}
for each choice of $x_{uwvv''}\in [uwvv''], x_{v'''}\in [v'''], x_{uwvw'v'}\in [uwvw'v'], x_u\in [u], x_v\in [v], x_{v'}\in [v']$,
where $C$ is defined in (\ref{aa1}).
Hence 
\begin{equation}\label{eqk1}
f_{m+n}(x_{uwvv''})\geq \frac{f_m(x_u)f_{n-k}(x_v)\min\{f_k(x): x\in X\}{D'}^2}{e^C M^2\max_{0\leq i\leq 2k} \{f_i(x): x\in X\}}.
\end{equation}
where $f_0(x)=1$ for all $x\in X$.
Define $$D'':=\min\{\frac{\min\{f_k(x): x\in X\}{D'}^2}{e^CM^2 \max_{0\leq i\leq 2k} \{f_i(x): x\in X\}}, D'\}.$$
Let $y_u\in [u]$. By (\ref{eqk1}) and the case when $p=k$
\begin{equation} 
\begin{split}
&\frac{\sum_{v\in F_n(u)}f_{m+n}(x_{uv})}{f_m(y_u)\sum_{r\in B_{n}(X)}f_{n}(x_{r})}\\
&\geq 
\frac{\sum_{v\in B_{n-k}(X), \vert w\vert<k}f_{m+n}(x_{uwvv''})+\sum_{v\in B_{n-k}(X), \vert w\vert=k}f_{m+n}(x_{uwv})}
{Mf_m(y_u)\sum_{r\in B_{n}(X)}f_{n}(x_{r})}  \label{inf}\\
&\geq
\frac{D''f_m(x_u)\sum_{v\in B_{n-k}(X)}f_{n-k}(x_v)}{Mf_m(y_u)\sum_{r\in B_{n}(X)}f_{n}(x_{r})}
\geq \frac{D''}{M^4e^CZ_k(\F)}
\end{split}
\end{equation}
where in the first inequality in (\ref{inf}), $w, v''$ are chosen for each $v\in B_{n-k}(X)$  in the manner described above.
Next we show the upper bound inequality in (\ref{RB}). 
For each $n,m\in\N$, let $u\in B_m(X)$ and $v\in F_n(u)$.
Then for each $x_{uv}\in [uv]$
\begin{equation*}
f_{m+n}(x_{uv})\leq f_m(x_{uv}) f_{n}(\sigma^m_X(x_{uv}))e^C
\leq \sup\{f_m(x): x\in [u]\} \sup\{f_{n}(x): x\in [v]\}e^C.
\end{equation*}
Hence
\begin{equation*}
\begin{split}
\frac{\sum_{v\in F_n(u)}f_{m+n}(x_{uv})}{f_m(x_u)\sum_{w\in B_{n}(X)}f_{n}(x_{w})} 
 &\leq \frac{\sup\{f_{m}(x): x\in [u]\}\sum_{v\in F_n(u)}\sup\{f_{n}(x): x\in [v]\} e^C}{f_m(x_u)\sum_{w\in B_{n}(X)}f_{n}(x_{w})}
\\ &\leq M^2e^C.
\end{split}
\end{equation*}
Hence  $X$ is right balanced with respect to $\F$. Similar arguments show that $X$ is 
left balanced with respect to $\F$.
\end{proof}

\section{Existence of Gibbs measures for sequences of continuous functions on subshifts} 
The main result of this section is Theorem \ref{main} on a necessary and sufficient condition for 
the existence of invariant Gibbs measures for sequences of continuous functions on subshifts and 
for the existence of Gibbs measures. 
Throughout this section, let $\F=\{\log f_n\}_{n\in\N}$ be a sequence of continuous functions on a one-sided subshift $X$.

\begin{lemma}\label{implication}
If $X$ is one-sided balanced with respect to $\F$, then $X$ is BSM($\F$).  
 \end{lemma}
\begin{proof}
We apply the proof found in \cite{kim}.
Assume that $X$ is right balanced to $\F$. 
In (\ref{RB}),  taking the sum  over all $x_{u}\in [u]$,$u\in B_m(X)$, we obtain that for $n\geq K$
\begin{equation*}
\begin{split}
\frac{1}{C}\sum_{u\in B_m(X)} \big[f_m(x_u) \sum_{w\in B_{n}(X)}f_{n}(x_{w}) \big]&\leq 
\sum_{u\in B_m(X)} [\sum_{v\in F_n(u)}f_{m+n}(x_{uv})]\\
&\leq C \sum_{u\in B_m(X)} \big[f_m(x_u) \sum_{w\in B_{n}(X)}f_{n}(x_{w}) \big].
\end{split}
\end{equation*}
The result follows immediately. 
\end{proof}

\begin{lemma}\label{charBSM}
A subshift  $X$ is BSM($\F$)  if and only if there exist $C\geq1$ and $P\in \R$ such that for each $n\in \N$ 
\begin{equation}\label{key1}
\frac{1}{C}\leq \frac{e^{nP}}{\sum_{u\in B_n(X)}f_n(x_u)}\leq C
\end{equation}
for each choice of  $x_u\in [u]$.
\end{lemma}
\begin{rem}\label{aboutP}
If $X$ is BSM ($\F$), we obtain $P=P(\F)$ by the proof below.
\end{rem}
\begin{proof}
 We modify the arguments found in  \cite[Lemma 3.11]{kim} slightly since we consider an arbitrary sequence $\F$.
We first show if part. The equation  (\ref{key1}) implies that  
\begin{equation*}
\frac{1}{C^2}\leq \frac{e^{(n+m)P}}{\sum_{u\in B_m(X)}f_m(x_u)\sum_{v\in B_n(X)}f_n(x_v)}\leq  C^2.
\end{equation*}
Applying (\ref{key1}) again by replacing $n$ by $n+m$, we obtain for each $m, n\in\N$
  \begin{equation*}
\frac{1}{C^3}\leq \frac{\sum_{w\in B_{m+n}(X)}f_{m+n}(x_w)}{\sum_{u\in B_{m+n}(X)}f_m(x_u)\sum_{v\in B_n(X)}f_n(x_v)}\leq  C^3.
\end{equation*}
Now we show the reverse direction. 
Let $m\in \N$ and $v\in B_m(X)$. Let $x_v, y_v\in [v].$  
Then there exists $d\geq 1$ such that  for $n\geq K$,  $K\in\N$
\begin{equation}\label{key2}
\frac{1}{d}\leq\frac{\sum_{u\in B_n(X)}f_{n}(x_{u})\sum_{v\in B_m(X)}f_m(x_{v})}{\sum_{w\in B_{m+n}(X)}f_{m+n}(x_{w})}\leq d
\end{equation}
for each choice of $x_u\in [u], x_w\in [w]$.
Hence fixing $n\geq K$
\begin{equation}\label{re}
\frac{1}{d^2}\leq\frac{\sum_{v\in B_m(X)}f_{m}(x_{v})}{\sum_{v\in B_m(X)}f_m(y_{v})}\leq {d}^2.
\end{equation}
Since each $[u]$, $u\in B_n(X),$ is a compact subset of $X$,  $\sup_{x\in [u]}f_{n}(x)$ attains a maximum at a point $\bar{x} \in [u]$. 
Taking such a point from each $[u]$, (\ref{re}) implies that for each $n\in\N$
\begin{equation}\label{key4}
\frac{1}{d^2} \sum_{u\in B_n(X)}f_{n}(y_{u})\leq  \sum_{u\in B_n(X)} \sup_{x\in [u]}f_{n}(x)\leq d^2{\sum_{u\in B_n(X)}f_n(y_{u})}.
\end{equation}
This implies that 
\begin{equation*} 
\frac{1}{d^6}\leq \frac{Z_m(\F)Z_n(\F)}{Z_{m+n}(\F)}\leq d^6.
\end{equation*}
for every $m\in\N$, $n\geq K$. Let $a_n=\log(d^6Z_n(\F))$. Then $a_{n+m}\leq a_n+a_m$  for every $n,m\geq  K$.
Modifying Fekete's subadditive lemma slightly to this case, we obtain that 
$$\lim_{n\to \infty}\frac{1}{n}\log (d^6Z_n(\F))=\inf_{n\geq K}\frac{1}{n}\log (d^6Z_n(\F))\leq \frac{1}{K}\log (d^6Z_K(\F)).$$ 
Since the sequence $\{\log (Z_n(\F)/d^6)\}_{n\geq K}$ satisfies $a_{n+m}\geq a_n+a_m$  for every $n,m\geq  K$,  we obtain\\
$$\lim_{n\to \infty}\frac{1}{n}\log (\frac{Z_n(\F)}{d^6})=\sup _{n\geq K}\frac{1}{n}\log (\frac
{Z_n(\F)}{d^6})\geq \frac{1}{K}\log (\frac{Z_K(\F)}{d^6}).$$
Thus $lim_{n\to \infty}a_n/n$ exists and define $P:=P(\F)=\lim_{n\to \infty}(1/n)
\log (d^6Z_n(\F)).$ Then 
$(1/n)\log (d^6Z_n(\F))\geq P$ for every $n\geq K$  and
$d^6Z_n(\F)\geq e^{nP}$ for every $n\geq K$. 
Using (\ref{key4}), we obtain 
$\sum_{u\in B_n(X)}f_n(x_u)\geq Z_n(\F)/{d^2}\geq e^{nP}/{d^8}$   for every  $n\geq K.$
Similarly, we obtain
$\sum_{u\in B_n(X)}f_n(x_u)\leq Z_n(\F)\leq d^6 e^{nP}$ for all $n\geq K.$
Hence (\ref{key1}) holds for $n\geq K$ by taking $C=d^8$.  Since 
$f_n$ is continuous on a compact set $X$, 
${e^{nP}}/{\sum_{u\in B_n(X)}f_n(x_u)}$ is bounded for $1\leq n \leq K-1.$

\end{proof}

 
\begin{thm}\label{main}
Let $\F=\{\log f_n\}_{n\in\N}$ be a sequence of continuous functions on a subshift $X$. Then
\begin{enumerate}[label=(\roman*)]
\item \label{m1}
There exists a Gibbs measure for $\F$ if and only if $X$ is right balanced with respect to $\F$. 
\item \label{m2}
There exists an invariant Gibbs measure for $\F$ if and only if $X$ is balanced with respect to $\F$.  If such an invariant measure $\mu$ exists for $\F$, then $P(\F)\in \R$ and 
\begin{equation} \label{vp}
P(\F)=h_{\mu}(\sigma_X)+\lim_{n\to \infty}\frac{1}{n}\int \log f_n d\mu.
\end{equation}
\end{enumerate}
\end{thm}
\begin{proof}
We first show \ref{m1}. Suppose there exists a Gibbs measure $\mu$ for $\F$ with a constant $C_0$. 
Hence,  for each $u\in B_m(X), m, n\in \N$, 
\begin{equation}\label{g1}
(1/C_0) e^{(n+m)P}\mu[u]\leq \sum_{v\in F_n(u)}f_{m+n}(x_{uv})\leq C_0e^{(n+m)P}\mu[u]
\end{equation}
 for each $x_{uv}\in [uv]$.
\begin{equation}\label{g2}
(1/C_0) e^{mP}\mu[u]\leq f_{m}(x_{u})\leq C_0e^{mP}\mu[u] \text{ for each } x_{u}\in [u].
\end{equation}
\begin{equation}\label{g3}
(1/C_0) e^{nP}\leq \sum _{v\in B_n(X)} f_{n}(x_{v})\leq C_0e^{nP} \text{ for each } x_{v}\in [v].
\end{equation}
Therefore by (\ref{g1}), (\ref{g2}), (\ref{g3})
\begin{equation*}
\frac{1}{C_0^3}\leq \frac{\sum_{v\in F_n(u)}f_{m+n}(x_{uv})}{f_m(x_u)
\sum _{v\in B_n(X)} f_{n}(x_{v}) }\leq C_0^3.
\end{equation*}
Hence $X$ is right balanced with respect to $\F$. 
For the other implication, 
suppose that  (\ref{RB}) holds for $n\geq K$ for some $K\in \N$ and for a constant $C$. 
Then $X$ is BSM($\F$). By the proof of Lemma \ref{implication}, for each $m\in\N, n\geq K$ we have  
\begin{equation}\label{use1}
\frac{1}{C}\leq\frac{\sum_{u\in B_m(X)}f_{m}(x_{u})\sum_{v\in B_n(X)}f_n(x_{v})}{\sum_{w\in B_{m+n}(X)}f_{m+n}(x_{w}
)}\leq C
\end{equation}
for each choice of $x_{u}\in[u]$, $x_{v}\in [v]$, $x_{w}\in [w]$, and there exists $d_1$ such that for each $n\in\N$
\begin{equation}\label{use2}
\frac{1}{d_1}\leq \frac{e^{nP}}{\sum_{u\in B_n(X)}f_n(x_u)}\leq d_1
\end{equation}
for each choice of  $x_u\in [u]$.
For all $n\in \N$, let $A_n$ be a set consisting exactly one point from each cylinder of length $n$ in X. 
Define the Borel probability measure $\nu_n$ on X concentrated on $A_n$ by
\begin{equation}
\nu_n= \frac{\sum_{x\in A_n}f_{n}(x)\delta_{x}}{\sum_{x\in A_n}f_n(x)}
\end{equation}
 where $\delta_{x}$ is the Dirac measure at $x$. 
 For 
 each $u\in B_n(X)$, denote by $\bar{x}_u$ a particular point chosen from $[u]$ in $A_n$.
Let $u\in B_m(X)$, $m\in\N$, be fixed.
For $n> m+K$, using (\ref{RB}), (\ref{use1}) and (\ref{use2}),
\begin{equation}\label{mea1}
\begin{split}
{\nu}_n([u])&=\frac{\sum_{v\in F_{n-m}(u)}f_{n}(\bar{x}_{uv})}{\sum_{w\in B_n(X)}f_n(\bar{x}_w)}
\leq \frac{Cf_{m}({x}_{u})\sum_{v\in B_{n-m}(X)}f_{n-m}({x}_{v})}{\sum_{w\in B_n(X)}f_n(\bar{x}_w)}\\
&\leq 
\frac{C^2f_{m}({x}_{u})}{{\sum_{w\in B_m(X)}f_m(\bar{x}_w)}}\leq C^2d_1f_m(x_u)e^{-mp}
\end{split}
\end{equation}
for each $x_u\in [u]$. Similarly, (\ref{RB}), (\ref{use1}) and (\ref{use2}) imply 
\begin{equation}\label{mea2}
{\nu}_n([u])\geq \frac{f_m(x_u)e^{-mp}}{C^2d_1}  \text{     for each } x_u\in [u]. 
\end{equation}
Let $\tilde{C}=C^2d_1.$
Taking a subsequence $\{\nu_{n_k}\}_{k\in\N}$ of $\{\nu_{n}\}_{n\in\N}$ converging to a measure $\nu$, letting $k\to \infty$, we obtain for each 
$x_u\in [u], u\in B_n(X), n\in \N$, 
\begin{equation}\label{gib}
\frac{1}{\tilde{C}}\leq \frac{\nu[u]}{e^{-np}f_n(x_u)}\leq \tilde{C}.
\end{equation}
Hence $\nu$ is a Gibbs measure for $\F$. 
Next we show \ref{m2}. Suppose there exists an invariant Gibbs measure $\mu$. It is easy to show 
the left balanced property by a proof similar to that of \ref{m1} by using the 
invariance of the Gibbs measure.  To show the reverse implication, 
let $\nu$ be the Gibbs measure in (\ref{gib}) constructed in the proof 
of \ref{m1}. Without loss of generality assume that (\ref{LB}) holds for 
$n\geq K$  and $C$ from (\ref{RB}). 
Since any weak limit point of $\{(1/n)\sum_{l=0}^{n-1}\nu\circ \sigma^{-l}_X\}_{n\in \N}$ is a $\sigma_X$-invariant
Borel probability measure,  let $\mu$ be the limit point of a convergent subsequence 
$\{(1/n_k)\sum_{l=0}^{n_k-1}\nu\circ \sigma^{-l}_X\}_{k\in \N}$.
Let $u\in B_m(X)$, $m\in\N$.
For $l>m+K$,  using (\ref{LB}), (\ref{use1}) and (\ref{use2}),
\begin{equation}\label{ugibbs}
\begin{split}
&\nu(\sigma^{-l}_X([u]))=\sum_{v\in P_l(u)}\nu([vu])\leq \tilde{C}e^{-(l+m)P}\sum_{v\in P_l(u)}f_{l+m}(x_{vu})\\
&\leq  \tilde{C}e^{-(l+m)P}Cf_m(x_u)\sum_{v\in B_l(X)}f_{l}(x_{v})
 \leq  \tilde{C}Cd_1e^{-mP}f_m(x_u).
 \end{split}
 \end{equation} 

\begin{equation}
\begin{split}
&\nu(\sigma^{-l}_X([u]))=\sum_{v\in P_l(u)}\nu([vu])\geq (1/\tilde{C})e^{-(l+m)P}\sum_{v\in P_l(u)}f_{l+m}(x_{vu})\\
&\geq  (1/(\tilde{C}C))e^{-(l+m)P}f_m(x_u)\sum_{v\in B_l(X)}f_{l}(x_{v})
 \geq  (1/(\tilde{C}Cd_1))e^{-mP}f_m(x_u).
 \end{split}
 \end{equation} 
Hence for each fixed $u\in B_{m}(X)$, $n > m+K$, 
\begin{equation*}
\frac{1}{n}\sum_{l=m+K+1}^{n-1}\frac{e^{-mP}f_m(x_u)}{\tilde{C}d_1C}\leq \frac{1}{n}
\sum_{l=0}^{n-1}\nu(\sigma^{-l}_X([u]))\leq \frac{m+K+1}{n}+\frac{1}{n}\sum_{l=m+K+1}^{n-1}
\tilde{C}d_1Ce^{-mP}f_m(x_u).
\end{equation*}
 Replacing $n$ by $n_k$ and letting $n_k\to \infty$, a simple computation shows that $\mu$ is an invariant Gibbs measure for $\F$.
Note that $P=P(\F)$ by Remark \ref{aboutP}.  We obtain (\ref{vp}) by the Gibbs property. 
Note that the Gibbs property of $\mu$ and bounded variation of $\F$ imply
that $\limsup_{n\to \infty}(1/n)\int \log f_n d\mu=\liminf_{n\to \infty}({1}/{n})\int \log f_n d\mu$.
 \end{proof}
 \begin{coro}
 A subshift $X$ is  balanced with respect to a sequence of continuous functions $\F$ if and only if 
$X$ is balanced with respect to $\F$ with $K=1$ in (\ref{RB}) and (\ref{LB}).
A subshift $X$ is right balanced with respect to $\F$ if and only if 
$X$ is right balanced with respect to $\F$ with $K=1$ in (\ref{RB}).
\end{coro}
\begin{proof}
The result is obvious by the first part of the proof of Theorem \ref{main} \ref{m1}  \ref{m2}.
\end{proof}

 \begin{coro}\label{main2}
 A subshift $X$ is right balanced with respect to 
 $f\in C(X)$ if and only if 
there exists a Gibbs measure for $f\in C(X)$.
 A subshift $X$ is balanced with respect to 
 $f\in C(X)$ if and only if 
there exists an invariant Gibbs measure for $f\in C(X)$.
\end{coro}
   
See \cite[Example 3.8]{kim} for an example of  a right balanced  two-sided subshift  with respect to a 
continuous function $f$ which belongs to the Bowen class, which is not left 
balanced.





\begin{coro}\label{nice}
Let $\mu\in M(X, \sigma_X)$. Then 
the measure $\mu$ is Gibbs for a sequence of continuous functions on a subshift $X$ if and only if 
$X$ is balanced with respect to the sequence $\G=\{\log \mu[x_1\dots x_n]\}_{n\in \N}$.\
\end{coro}
\begin{proof}
This is obvious because $\mu$ is Gibbs for $\G$. 
\end{proof}
\begin{coro}\label{onesidedGibbs}
Suppose that a subshift  $X$ is BSM($\F$). If 
there exist $C \geq 1$ and $K\in \N$ such that for each $m\in \N, n\geq K,
u\in B_m(X),$ the upper bound (resp. lower bound) inequalities in (\ref{RB}) and (\ref{LB}) hold, then  
there exists an invariant measure $\mu$ which satisfies the upper bound (resp. lower bound) property of (\ref{gibbsd}).
\end{coro}
\begin{proof}
Since $X$ is BSM($\F$), we apply the proof of Theorem \ref{main} to construct an invariant measure $\mu$
satisfying the upper bound (lower bound) property of (\ref{gibbsd}). 
\end{proof}


 
\section{Theorem \ref{main} for subadditive and superadditive sequences} \label{onlysub}
In this section, we study Theorem \ref{main} for subadditive and superaddive sequences $\F$. In particular, 
if $\F$ is superadditive on a subshift with the strong specification,  
an invariant Gibbs measure for $\F$ is a unique invariant Gibbs measure and it is ergodic. We also give examples which illustrate Theorem \ref{main}.
In the following theorems, let $\F=\{\log f_n\}_{n\in\N}$ be a sequence of continuous functions on 
a one-sided subshift $X$. If $\F$ has bounded variation, let $M$ be a number such that
$\sup_{n\in\N} \{ {f_n(x)}/{f_n(y)} : x,y  \in X, x_i=y_i \textrm{ for }1 \leq i \leq n\}\leq M$.

\begin{thm}\label{mishiro}
Suppose $\F$ is subadditive on a subshift $X$. Then  
\begin{enumerate}[label=(\roman*)]
\item \label{s1}
$X$ is right balanced with respect to $\F$ if and only if  $\F$ has bounded variation and there exist $C_1$, $0<C_1<1$ and $k\in\N$ such that  for every $n, m\in \N$, $u\in B_m(X)$,  
\begin{equation}\label{ss1}
\sum_{v\in F_{n+k}(u)}f_{m+n+k}(x_{uv})\geq C_1 f_{m}(x_u)\sum_{v\in B_n(X)}f_{n}(x_v)
\end{equation}
for each choice of  $x_{uv}\in [uv]$, $x_u\in [u]$ and $x_v\in [v]$. 
\item \label{s2}
$X$ is left balanced with respect to $\F$ if and only if $\F$ has bounded variation and there exist $C_1$, $0<C_1<1$ and $k\in\N$ such that  for every $n, m\in \N$, $u\in B_m(X)$,  
\begin{equation}\label{ss2}
\sum_{v\in P_{n+k}(u)}f_{m+n+k}(x_{vu})\geq C_1 f_{m}(x_u)\sum_{v\in B_n(X)}f_{n}(x_v)
\end{equation}
for each choice of  $x_{vu}\in [vu]$, $x_u\in [u]$ and $x_v\in [v]$. 
\item \label{s3}
$X$ is balanced with respect to $\F$ if and only if $\F$ has bounded variation and equations (\ref{ss1}) and 
(\ref{ss2}) hold.
\item \label{s4}
If $\F$ has bounded variation and the sequence $\{\log Z_n(\F)\}_{n\in \N}$ is superadditive, 
then there exists an equilibrium state $\mu$ for $\F$ satisfying the upper bound property of (\ref{gibbsd}).
\end{enumerate}
\end{thm}
\begin{proof}
Recall that if $X$ is one-sided balanced with respect to $\F$, then it has bounded variation. We show \ref{s1}. If $X$ is right balanced with respect to $\F$, define $C\geq 1$ and  $k=K\in\N$ from (\ref{RB}).  Then, for every $n, m\in \N$, $u_m\in B_m(X)$,  
$\sum_{v\in F_{n+k}(u)}f_{m+n+k}(x_{uv})\geq (1/C) f_{m}(x_u)\sum_{w\in B_{n+k}(X)}f_{n+k}(x_w)$
for each choice of  $x_{uv}\in [uv]$, $x_u\in [u]$ and $x_w\in [w]$.  By the proof of Lemma \ref{implication}, for each $n\geq 1$
\begin{equation*} 
C{\sum_{w\in B_{n+k}(X)}f_{n+k}(x_{w})}\geq   \sum_{v\in B_n(X)} f_n(x_{v}) \sum_{r\in B_k(X)}f_{k}(x_{r})\geq \frac{Z_{k}(\F)\sum_{v\in B_n(X)} f_n(x_{v})}{M}
\end{equation*} 
where $x_w\in [w]$, $x_v\in [v]$, and  $x_{r}\in [r]$.
Hence  
\begin{equation*}
\frac{1}{C}f_{m}(x_u)\sum_{w\in B_{n+k}(X)}f_{n+k}(x_w)\geq \frac{Z_{k}(\F)}{MC^2}f_{m}(x_u)\sum_{v\in B_n(X)}f_n(x_{v}).
\end{equation*}
To show the reverse implication, first notice that the upper bound inequality of (\ref{RB}) holds (see at the end of the poof of
Proposition \ref{submulti}).
Since $\F$ is subadditive, define $C$ from (\ref{aa1}). Now take $n\geq k+1$ and write $n=l+k, l\in\N$. Let $y_u\in [u]$.
Then 
\begin{equation*} 
\begin{split}
&\frac{\sum_{v\in F_{n}(u)}f_{m+n}(x_{uv})}{f_m(x_u)\sum_{v\in B_n(X)} f_n(x_{v})}=
\frac{\sum_{v\in F_{l+k}(u)}f_{m+l+k}(x_{uv})}{f_m(x_u)\sum_{v\in B_{l+k}(X)} f_{l+k}(x_{v})}\\
&\geq \frac{C_1f_m(y_u)\sum_{v\in B_{l}(X)}f_{l}(x_{v})}{f_m(x_u)\sum_{v\in B_{l+k}(X)} f_{l+k}(x_{v})}
\geq \frac{C_1}{e^CM^2Z_k(\F)}
\end{split}
\end{equation*}
where in the last inequality we use the subadditive property of $\F$. Now set $K=k+1$ in (\ref{RB}).
Similar arguments show \ref{s2}. \ref{s3} follows immediately.  To see \ref{s4}, first note 
that since $\F$ is subadditive, the following variational principal holds \cite{CFH}: 
\begin{equation*}
P(\F)=\sup_{\mu\in M(X,\sigma_X)}\{h_{\mu}(\sigma_X)+\lim_{n\to \infty}\frac{1}{n}\int \log f_n d \mu\}.
\end{equation*}
Since $\F$ is subadditive with bounded variation, by the final part of the proof of Proposition \ref{submulti},
the upper bound property of (\ref{RB}) holds. Similarly,  the upper bound property of (\ref{LB}) holds for every $n,m\in\N$. It is easy to see that the lower bound property of (\ref{bms}) also holds for every $n,m\in\N$. Since $\{\log Z_n(\F)\}_{n\in \N}$ is superadditive and $\F$ has bounded variation, $X$ is BSM($\F$). The proof of Theorem \ref{main} implies that there exists a measure $\nu$ satisfying the upper bound property in (\ref{gib}). Hence (\ref{ugibbs}) holds for each $l>m$ and by the proof of Theorem \ref{main} there exists an invariant measure $\mu$ satisfying the upper bound property of (\ref{gibbsd}).
Noting that $P=P(\F)$ in (\ref{gibbsd}) (see Remark \ref{aboutP}), we obtain that $P(\F)\leq h_{\mu}(\sigma_X)+\lim_{n\to \infty}(1/n)\log f_n d\mu$. 
 
\end{proof}

Next we characterize quasi-multiplicative sequences using the balanced property.
\begin{prop}\label{ncl}
Let $\F$ be a subadditive sequence on a subshift $X$ with bounded variation. 
Then $\F$ satisfies \ref{a1} in which $B_i(X)$ is replaced by $B_k(X)$ where $k$ is defined in \ref{a1} if and only if
$X$ is right balanced with respect to $\F$ and 
there exists $\tilde{C}>0$ such that  for every $n, m\in \N$, $u\in B_m(X)$,  $v\in B_n(X)$, $k$ defined from (\ref{ss1}),
\begin{equation}\label{chq}
\frac{f_n(x_v)\sum_{w\in F_{n+k}(u)}f_{m+n+k}(x_{uw})}{ \sum_{w\in B_n(X)}f_{n}(x_w)} \leq \tilde{C}
\sum_{\substack{uCv\in B_{m+n+k}(X) \\ C\in B_{k}(X)}} f_{m+n+k}(x_{uCv})
\end{equation}
for each choice of  $x_v\in [v], x_{w}\in [w], x_{uw}\in [uw]$, $x_{uCv}\in [uCv]$. 
\end{prop}
\begin{rem}
The right balanced property of $X$ with 
respect to $\F$ and (\ref{chq})  imply that $X$ is left balanced with 
respect to $\F$. 
\end{rem}
\begin{proof}
Suppose $X$ is a subshift over $S$ symbols. By (\ref{ss1}) and (\ref{chq}), we obtain 
\begin{equation*} 
\begin{split}
\sum_{\substack{uCv\in B_{m+n+k}(X)\\C\in B_{k}(X)}} f_{m+n+k}(x_{uCv})&\geq 
\frac{f_n(x_v)\sum_{w\in F_{n+k}(u)}f_{m+n+k}(x_{uw})}{ \tilde{C}\sum_{w\in B_n(X)}f_{n}(x_w)} \\
&\geq \frac{C_1}{\tilde{C}M} {f_{m}(x_u)f_n(x_v)}.
\end{split}
\end{equation*} 
Since the maximum number of the allowable words of length $k$ is $S^k$, $\F$ satisfies \ref{a1} in which $B_i(X)$ is replaced by   $B_k(X)$. 
For the reserve implication, define $k$ from \ref{a1}. 
Then clearly (\ref{ss1}) holds. 
Define $C$ from (\ref{aa1}). Since $\F$ is subadditive and has bounded variation, it is easy to see that 
\begin{equation*} 
\begin{split}
\frac{f_n(x_v)\sum_{w\in F_{n+k}(u)}f_{m+n+k}(x_{uw})}{\sum_{w\in B_n(X)}f_{n}(x_w)}\leq  M^2Z_k(\F)e^{2C}f_m(x_u)f_n(x_v)
\end{split}
\end{equation*}
or each choice of $x_v\in [v], x_u\in [u], x_{uw}\in [uw], 
x_w\in[w]$. Now the result follows. 
 \end{proof}


\begin{thm}\label{mishiroU}
Suppose $\F$ is superadditive on a subshift $X$ satisfying the strong specification property with a specification number $k\in \N$. Then  
\begin{enumerate}[label=(\roman*)]
\item \label{sp1}
$X$ is right balanced with respect to $\F$ if and only if $\F$ has bounded variation and there exists $C_1\geq 1$ such that  for every $n, m\in \N$, $u\in B_m(X)$,  
\begin{equation}\label{sup1}
\sum_{v\in F_{n+k}(u)}f_{m+n+k}(x_{uv})\leq C_1 f_{m}(x_u)\sum_{v\in B_n(X)}f_{n}(x_v)
\end{equation}
for each choice of  $x_{uv}\in [uv]$, $x_u\in [u]$ and $x_v\in [v]$. 
\item \label{sp2}
$X$ is left balanced with respect to $\F$ if and only if $\F$ has bounded variation and
there exists $C_1\geq 1$ such that  for every $n, m\in \N$, $u\in B_m(X)$,  
\begin{equation}\label{sup2}
\sum_{v\in P_{n+k}(u)}f_{m+n+k}(x_{vu})\leq C_1 f_{m}(x_u)\sum_{v\in B_n(X)}f_{n}(x_v)
\end{equation}
for each choice of  $x_{vu}\in [vu]$, $x_u\in [u]$ and $x_v\in [v]$. 

\item \label{sp3}
$X$ is balanced with respect to $\F$ if and only if $\F$ has bounded variation and equations (\ref {sup1}) and (\ref {sup2}) hold.
\item \label{sp4}
$X$ is balanced with respect to $\F$ if and only if there exists a unique invariant Gibbs measure
for $\F$. 
\item \label{sp5}
If $X$ is balanced with respect to $\F$, then the variational principal holds:
\begin{equation*}
P(\F)=\sup_{\mu\in M(X,\sigma_X)}\{h_{\mu}(\sigma_X)+\lim_{n\to \infty}\frac{1}{n}\int \log f_n d \mu\},
\end{equation*}
and the unique invariant Gibbs measure
for $\F$ is an ergodic equilibrium state for $\F$.  
 \end{enumerate}
\end{thm}
\begin{rem}
In \cite[Lemma 2.6]{R}, it was shown that an invariant Gibbs measure for a superadditive sequence $\F$, if 
it exists, is an equilibrium state for $\F$. 
In the theory of matrix cocyles, Wu \cite{w1} showed the variational principle for superadditive 
sequences of continuous functions under a certain setting.
\end{rem}
\begin{proof}
We show \ref{sp1}. Recall that if $X$ is right balanced with respect to $\F$, then 
(\ref{RB}) holds for $K=1$. 
Hence similar arguments as in the proof of Theorem \ref{mishiro} \ref{sp1} show the result.
For the other implication, we first note that 
for any $u, v\in B(X)$, 
there exists $w\in B_{k}(X)$ such that $uwv\in B(X)$. 
We first show that there exists 
$\tilde{C}>0$ such that $Z_{n+m}(\F)\leq \tilde{C} Z_{n}(\F)Z_{m}(\F)$ for each $n,m\in\N$. Taking the sum 
over all $u\in B_m(X)$ in (\ref {sup1}), we obtain 
\begin{equation*}
\sum_{w\in B_{m+n+k}(X)}f_{m+n+k}(x_{w})\leq C_1\sum_{u\in B_m(X)} f_{m}(x_u)\sum_{v\in B_n(X)}f_{n}(x_v)
\leq C_1M^{2}Z_m(\F)Z_n(\F)
\end{equation*}
for each $x_{w}\in [w]$, $x_u\in [u]$ and $x_v\in [v]$. 
By the superadditivie property  in (\ref{aa2}), for $w\in B_{m+n+k}(X)$
\begin{equation}\label{sub1}
\begin{split}
f_{m+n+k}(x_{w})
&\geq e^{-C} f_{m+n}(x_{w})f_{k}(\sigma^{m+n}_Xx_{w})
\\
&\geq \frac{\min\{f_k(x): x\in X\}}{e^{C}M}\sup \{f_{m+n}(x): x\in [w_1\dots w_{m+n}] \}.
\end{split}
\end{equation}
Given $w_1\dots w_{m+n}\in B_{m+n}(X)$, there exists $\bar w\in B_k(X)$ such that $w_1\dots w_{m+n}\bar w\in B_{m+n+k}(X)$. Hence
summing over all allowable words of length $(m+n)$ in $X$ in (\ref{sub1}),
\begin{equation*}
\sum_{w\in B_{m+n+k}(X)}f_{m+n+k}(x_{w})
\geq  \frac{\min\{f_k(x): x\in X\} Z_{m+n}(\F)}{e^{C}M}.
\end{equation*}
Now we show the lower bound inequality in (\ref{RB}). 
For each $n,m\in\N$, fix $u\in B_m(X)$ and let $v\in F_{n+k}(u)$.
Given $\bar{v}\in B_n(X)$, there exists $w\bar{v}\in F_{n+k}(u)$ for some $w\in B_k(X)$. 
For $x_{uw\bar{v}}\in [uw\bar{v}]$
\begin{equation*} 
\begin{split}
f_{m+n+k}(x_{uw\bar{v}}) 
 &\geq e^{-2C}f_{m}(x_{uw\bar{v}})f_{k}(\sigma^{k}_Xx_{uw\bar{v}})f_{n}(\sigma^{m+k}_Xx_{uw\bar{v}})\\
 &\geq \frac{\min \{f_{k}(x): x\in X\}\sup\{f_m(x):x\in [u]\}\sup\{f_{n}(x): x\in [\bar{v}]\}}{e^{2C}M^2}.
  \end{split}
\end{equation*}
Hence 
 \begin{equation*} 
\begin{split}
\sum_{v\in F_{n+k}(u)}f_{m+n+k}(x_{uv}) &\geq 
\frac{\sum_{\bar{v}\in B_n(X), uw\bar{v}\in B_{m+n+k}(X)}f_{m+n+k}(x_{uw\bar{v}})}{M}\\
 &\geq \frac{\min \{f_{k}(x): x\in X\}\sup\{f_m(x):x\in [u]\} Z_n(\F)}{e^{2C}M^3}.\\
 \end{split}
\end{equation*}
Using $Z_{n+k}(\F)\leq \tilde{C} Z_{n}(\F)Z_{k}(\F)$, for each $n\in\N$
\begin{equation*} 
\begin{split}
 \frac{\sum_{v\in F_{n+k}(u)}f_{m+n+k}(x_{uv})}{f_{m}(x_{u}) \sum_{v\in B_{n+k}(X)}f_{n+k}(x_{v})}
 &\geq \frac{\sup\{f_m(x):x\in [u]\} \min \{f_{k}(x): x\in X\} Z_n(\F)}{f_{m}(x_{u})e^{2C}M^4Z_{n+k}(\F)}\\
 &\geq \frac{\min \{f_{k}(x): x\in X\}}{M^5\tilde{C}e^{2C}Z_k(\F)}
 \end{split}
\end{equation*}
 for each $x_{uv}\in [uv]$, $x_u\in [u]$ and $x_v\in [v]$. 
  Now we show the upper bound inequality in (\ref{RB}). It is easy to see that for each $k_1,k_2\in \N$, 
 $$Z_{k_1+k+k_2}(\F)\geq \frac{\min\{f_k(x): x\in X\} Z_{k_1}(\F) Z_{k_2}(\F)}{e^{2C}M^2}.$$
 Set $A:=\min\{f_k(x): x\in X\}/(e^{2C}M^2)$. Then for $k_1,k_2\in \N$, 
 \begin{equation*}
 Z_{k_1+k_2}(\F)\leq \tilde{C}Z_{k_1}(\F) Z_{k_2}(\F)\leq \frac{\tilde{C}}{A}Z_{k_1+k_2+k}(\F).
 \end{equation*}
 Let $A'=Z_1(\F)/Z_{k+1}(\F)$ and define $B:=\max\{A', {\tilde{C}}/{A}\}$. 
 Then for $n\geq 1$, 
 \begin{equation*} 
\begin{split}
 \frac{\sum_{v\in F_{n+k}(u)}f_{m+n+k}(x_{uv})}{f_{m}(x_{u}) \sum_{v\in B_{n+k}(X)}f_{n+k}(x_{v})}
 &\leq \frac{C_1f_m(y_u)\sum_{v\in B_{n}(X)}f_{n}(x_{v})}
 {f_{m}(x_{u})\sum_{v\in B_{n+k}(X)}f_{n+k}(x_{v})}\\
 &\leq C_1M^2\frac{Z_n(\F)}{Z_{n+k}(\F)}\leq C_1M^2B.
 \end{split}
\end{equation*}
Finally set $K=k+1$ in (\ref{RB}). 
 Similar arguments show \ref{sp2}. \ref{sp3} follows immediately. For \ref{sp4}, suppose $X$ is balanced with 
 respect to $\F$.  We prove that an invariant Gibbs measure $\mu$ for $\F$ is ergodic by showing that  
 there exists $A>0$ such that for each fixed $u\in B_{m}(X)$ and $v\in B_{n}(X)$, we have $\mu([u]\cap \sigma^{-l}_X([v]))
 \geq {A}\mu([u])
\mu([v])$ for all $l> m+2k$.  

 \textbf{[Claim]} 
There exists $D>0$ such that for fixed allowable words 
$u \in B_m(X), v \in B_n(X)$, $m,n,t\in\N$, we have
\begin{equation*}
\begin{split}
&\sum_{ua_1\dots a_{t+2k}v\in B_{m +n+t+2k}(X)} \sup \{f_{m+n+t+2k}(x): x \in [ ua_1\dots a_{t+2k}v]\}\\
& \geq D^2\sup\{f_m(x): x\in [u]\}\sup\{f_n(x): x\in [v]\}Z_{t}(\F).
\end{split}
\end{equation*}
We show the claim at the end of the proof. 
Let $M_1=\max\{0, P\}$ and $\bar C$ be a constant from the Gibbs property in (\ref{gibbsd}). 
Let $\tilde{C} $ be defined from
(\ref{key1}). For $l> m+2k$, set $t=l-m-2k$, using the claim above,
\begin{align*}
&\mu([u]\cap\sigma^{-l}_X([v]))
=\sum_{ua_1\dots a_{t+2k}v\in B_{n+l}(X)}\mu([ua_1\dots a_{t+2k}v])\\
&\geq  \frac{e^{-(m+n+t)P-2kM_1}}{\bar{C}M} 
\sum_{ua_1\dots a_{t+2k}v\in B_{n+l}(X)}
 \sup\{f_{m+n+t+2k}(x): x\in [ua_1\dots a_{t+2k}v]\}\\
&\geq  \frac{ D^2e^{-(m+n+t)P-2kM_1}}{\bar{C}M} 
Z_t(\F)\sup\{f_{m}(x): x\in [u]\}\sup\{f_{n}(x): x\in [v]\}\\
&\geq  \frac{D^2e^{-2kM_1}}{M\bar{C}^3\tilde {C}} \mu([u])\mu([v]).
\end{align*}

The Gibbs property with ergodicity implies that the uniqueness of
invariant Gibbs measures for $\F$. 
Now we show the claim. Let $m_{k}:=\min\{f_{k}(x): x\in X\}$. 
For a fixed $t\in \N$, fix $c\in B_{t}(X)$. Given $v$ and $c$,  there exists $w_1\in B_{k}(X)$ such that 
\begin{equation}\label{k0}
\sup \{f_{t+k +n}(x): x \in [cw_1v]\} \geq \frac{m_k\sup\{f_t(x): x\in [c]\}\sup\{f_n(x): x\in [v]\}}{M^2}.
\end{equation}
Define $D_1:=m_{k}/M^2$. For fixed $u$ and $cw_1v$ above, there  exists 
$w_2\in B_{k}(X)$ such that 
\begin{align}
&\sup \{f_{m+t+n+2k}(x): x \in [uw_2cw_1v]\} \label{k2}\\ 
& \geq D_1\sup\{f_m(x): x\in [u]\}\sup \{f_{t+k +n}(x): x \in [ cw_1v]\} \label{k3}\\
 & \label{k4}  \geq D_1^2\sup\{f_m(x): x\in [u]\}\sup\{f_t(x): x\in [c]\}\sup\{f_n(x): x\in [v]\}. 
\end{align}
Summing over all allowable words $c\in B_{t}(X)$, each of which satisfies (\ref{k0}) and (\ref{k2})-(\ref{k4}) with some $w_1, w_2$, 
we obtain the claim.
Finally we show \ref{sp5}. First note that  the unique invariant Gibbs measure is ergodic by the proof above and satisfies (\ref{vp}). Now we apply a proof found in \cite{b2}. 
For completeness, we give a sketch of the proof. First note that
for any $p_i\geq 0, \sum_{i=1}^{l}p_i=1, c_i\in \R$, 
$$\sum_{i=1}^{l}p_i(-\log p_i+c_i)\leq \log \sum_{i=1}^{l}e^{c_i}.$$
For $u\in B_n(X)$, $n\in \N$, set $C_{u}=\sup\{f_n(x): x\in [u]\}$. Then 
$Z_n(\F)=\sum_{u\in B_n(X)} C_{u}$.
Then for each $\mu\in M(X, \sigma_X)$
$$\sum_{u\in B_n(X)}\mu[u](\log C_u-\log\mu[u]-\log Z_n(\F))\leq 0.$$
 Dividing by $n$ and letting $n\to \infty$, we obtain
 $P(\F)\geq h_{\mu}(\sigma_X)+\lim_{n\to \infty}(1/{n})\int \log f_n d \mu$.
 Note that $\lim_{n\to \infty}({1}/{n})\int \log f_n d \mu$ exists because $\G:=\{\log (1/f_n)\}_{n\in \N}$ 
 is subadditive and $\log (1/f_1)^{+} \in L_1(\mu)$.
 Taking the supremum over all $\mu\in M(X, \sigma_X)$ and using the existence of an invariant Gibbs measure 
 satisfying (\ref{vp}), we  obtain the results.

 \end{proof}

\begin{coro} 
If $X$ is balanced with respect to a sequence of continuous functions $\F$ on a subshift $X$ satisfying
 \ref{a1} in which  $B_i(X)$ is replaced by $B_k(X)$, then there exists a unique invariant Gibbs measure
for $\F$ and it is an ergodic measure.  
\end{coro}
\begin{proof}
By the proof of Theorem \ref{mishiroU} \ref{sp4},
 if $\F$ satisfies \ref{a1} in which  $w\in B_k(X)$ for each $w$, 
  any invariant Gibbs measure for $\F$ is ergodic. 
\end{proof}

\begin{ex}\label{ex1}
Let $p_1:=(1,2,3,1,2,3..)$, $p_2:=(2,3,1,2,3..), p_3:=(3,1,2,3,1,2,3..) \in \Sigma^{+}_{3}$ be points of period 3.
Let $p_4:=(2,2,2\dots)\in \Sigma^{+}_{3}$ be a point of period 1. 
%
Let $X=\{p_1,p_2,p_3,p_4\}$. Then $(X, \sigma_X)$ is not irreducible. Define $\bar{\mu}\in M(X,\sigma_X)$ by $\bar{\mu}(\{p_i\})=1/4$ for $1\leq i\leq 4$. Let $\phi_n(x):=\bar{\mu}[x_1\dots x_n]$ for each $x \in X$, $n\in \N$  and define $\Phi:=\{\log \phi_n\}_{n \in \N}$. Then $\Phi$ is almost additive, 
not quasi-multiplicative, and $\bar{\mu}$ is  Gibbs  for $\Phi$. In fact, $X$ is balanced respect to $\Phi$ because
\begin{equation}\label{m}
 \frac{\sum_{v\in F_n(u)}\phi_{m+n}(x_{uv})}{\phi_m(x_u)\sum_{w\in B_{n}(X)}\phi_{n}(x_{w})}=
\frac{\sum_{v\in P_n(u)}\phi_{m+n}(x_{vu})}{\phi_m(x_u)\sum_{w\in B_{n}(X)}\phi_{n}(x_{w})}=\frac{\bar {\mu}[u]}{\bar{\mu}[u]}=1
\end{equation}
for each choice of 
$x_{u}\in[u]$, $x_{uv}\in [uv]$, $x_{vu}\in [vu]$ $x_{w}\in [w]$ 
Note that $\bar\mu$ is not a unique invariant Gibbs measure for $\Phi$ and is not ergodic.

\end{ex}
\begin{ex}\label{ex2}
Let $(X, \sigma_X)$ be a subshift. Then  $X$ is balanced with respect to $\Phi:=\{\log \mu[x_1\dots x_n]\}_{n\in \N}$ for any $\mu\in M(X, \sigma_X)$.   If  $\mu$ is a Borel probability measure on $X$, 
$X$ is right balanced with respect to $\Phi$.

\end{ex}
\begin{ex}\label{ex3} 
In the study of matrix cocycle potentials, 
Rush \cite[Theorem 1.1]{R} showed the existence and uniqueness of an invariant Gibbs measure  for a superadditive sequence 
$\Phi_t=\{\phi_{t,n}\}_{n\in\N}, t<0$, 
on a one-sided topologically 
mixing shift of finite type $X$ and its ergodicity when $t$ is sufficiently close to $0$,  
where 
$\phi_{t,n}:=S_n\psi +t \log f_n$, where $\psi$ is H{\"o}lder continuous  and  $\{\log f_n\}_{n\in\N}$ is 
a quasi-multiplicative sequence associated with a cocycle under some conditions 
(See \cite{R} for details). 
Theorem \ref{main} (Theorem \ref{mishiroU}) implies that $X$ is balanced with respect to 
such sequence $\Phi_t, t<0, $ where $t$ is sufficiently close to $0$ (See also  \cite{MV, MQ} for some examples of Gibbs measures for superadditve sequences).
\end{ex}
\begin{ex}
We give an example of Corollary \ref{onesidedGibbs}. Take a sequence $\F$ on a subshift $X$ defined in Theorem \ref{past1}. 
In general $\F$ is neither subadditive nor superadditive. It is easy to see that  $\{\log Z_n(\F)\}_{n\in\N}$ is almost additive and $\F$ has bounded variation. Hence $X$ is BSM($\F$). By the proof of Theorem \ref{general}, $X$ satisfies the lower bound inequalities in (\ref{RB}) and (\ref{LB}). Hence there exists an invariant measure 
satisfying the lower bound property of (\ref{gibbsd}) for $\F$.
\end{ex}

\section{Applications}\label{apli}
In this section, we apply Theorem \ref{main} to study one-block factor maps $\pi: X \to Y $ between subshifts. 
Given a unique equilibrium 
state $\mu$ for a subadditive sequence $\G$ on $Y$ which is also Gibbs for $\G$, we  consider
preimage measures $\nu$ on $X$ such that $\nu$ is a measure of maximal relative entropy over $\mu$.  
We study a necessary and sufficient condition for such a measure to be an invariant Gibbs measure.

Let  $(X, \sigma_X)$ and $(Y, \sigma_Y)$ be one-sided subshifts. 
A map $\pi:X\rightarrow Y$ is a  factor map if it is continuous, surjective and satisfies 
$\pi \circ \sigma_{X} = \sigma_Y\circ \pi$. A one-block code  is a map $\pi:X\rightarrow Y$ for which there exists a function, denoted again by $\pi$, $\pi:B_1(X)\rightarrow B_1(Y)$ such that 
$(\pi(x))_i=\pi(x_i)$ for all $i\in\N$.
Given a one-block factor map $\pi: X\rightarrow Y$ between subshifts and 
an invariant measure $\mu$ on $X$, define the image $\pi\mu\in M(Y, \sigma_Y)$ by 
$\pi\mu(B):= \mu(\pi^{-1}B)$ { for a Borel set  $B$ of $Y$}. 
If $X$ has the strong specication property, for $y=(y_1, \dots, y_n, \dots)\in Y$,
define  $\vert \pi^{-1}[y_1\dots y_n]\vert $ to be the cardinality of the allowable words of length $n$ in $X$ such that $\pi(x_1\dots x_n)=y_1\dots y_n$. For $n\in\N$ and  $y\in Y$, define
$\phi_n(y):=\vert \pi^{-1}[y_1\dots y_n]\vert$. Then $\{\log \phi_n\}_{n\in\N}$ is subadditive with $C=0$ in (\ref{aa1}) and 
satisfies \ref{a1} with bounded variation \cite{Fe, Y2}. The next theorem extends  \cite[Corollary 3.4]{w} to sequences of continuous functions.

\begin{thm} \label{past1} (see \cite[Lemma 3.2]{ly}, \cite[Theorem  3.8]{Y2})
Let $\pi: X \to Y$ be a one-block factor map between subshifts where 
$X$ has the strong specification property. Suppose that there exists an invariant Gibbs measure for 
a subadditive sequence $\G=\{\log g_n\}_{n\in \N}$ on $Y$ . For each $n\in\N$, let ${f}_n:=\frac{g_n\circ\pi}{\phi_n\circ \pi}$ 
and $\F=\{\log f_n\}_{n\in\N}$. Then
\begin{equation*}
\begin{split}
P(\F)=&\sup_{m\in M(X, \sigma_X)}\{h_{m}(\sigma_X)+\lim_{n\to \infty}\frac{1}{n}\int \log \big(\frac{g_n\circ\pi}{\phi_n\circ \pi}\big)dm\}\\
&=\sup_{m\in M(Y, \sigma_Y)}\{h_{m}(\sigma_Y)+\lim_{n\to \infty}\frac{1}{n}\int \log g_n dm\}=P(\G).
\end{split}
\end{equation*}
There exists an equilibrium state $\nu$ for ${\F}$. A measure $\nu$ is an equilibrium state for ${\F}$ if and only if
 $\nu$ is a measure of maximal relative entropy over $\pi\nu$, i.e., 
$h_{\nu}(\sigma_X)=\max\{h_{\bar{\nu}}(\sigma_X): \bar \nu\in M(X, \sigma_X), \pi \bar{\nu}=\pi\nu\}$, and $\pi\nu$ is 
an equilibrium state for ${\G}$.

\begin{coro}\label{ucase} In Theorem \ref{past1}, suppose an invariant Gibbs measure $\mu$ for $\G$ is a unique equilibrium state 
for $\G$. Then  a measure $\nu$ on $X$  is an equilibrium state  for ${\F}$ if and only if $\nu$ is a measure  of maximal relative entropy over $\mu$.
\end{coro} 



\end{thm}
In general, $\F$ in  Theorem \ref{past1} is neither subadditive nor superadditive.  If we set $g_n={\phi_n}^{1/(1+\alpha)}$,
$\alpha>0$, there is a unique invariant Gibbs measure $\G$ which is also a unique equilibrium state for $\G$ 
and  this sequence appears in \cite{Fe, Y2}. By \cite{Fe} $\nu$ is  unique and ergodic. 
The Gibbs property in general is not known.   
The next theorem gives a necessary and sufficient condition for such a measure  $\nu$ 
to be Gibbs for a sequence of continuous functions. 

\begin{thm} \label{general}
In Theorem \ref{past1} and Corollary \ref{ucase},
an equilibrium state $\nu$ for $\F$ is an invariant Gibbs measure for $\F$  if and only if
there exist $C \geq 1$ and $K\in \N$ such that for each $m\in \N, n\geq K,
u\in B_m(X),$ the upper bound inequalities in (\ref{RB}) and (\ref{LB}) hold.
\end{thm}

\begin{proof}
It is enough to show that $\F=\{\log f_n\}_{n\in\N}$ satisfies the lower bound inequalities in (\ref{RB}) and (\ref{LB}). 
Since $\G$ has a Gibbs measure, it has bounded variation and there exists $M\geq 1$ such that
$ g_n(x)/g_n(y)\leq M$ for all $x, y\in [u], u\in B_n(Y), n\in \N$.
Since $X$ has the specification property, let $k$ be a number for specification. Suppose that $X$ is a subshift over $l$ symbols, $l\geq 2$. 
 Let $u\in B_m(X)$ be fixed and let  $n>k$.  
 Since for any $v\in B_{n-k}(X)$ there exists $w\in B_k(X)$ such that $uwv\in B_{m+n}(X)$, 
 let $W$ be the set consisting of all possible such $w\in B_k(X)$ and suppose 
$W=\{w_1, w_2, \dots, w_N\}$ for some $N\in\N$.
 First note that 
\begin{equation}
M\sum_{v\in F_n(u)} f_{m+n}(x_{uv})
\geq \sum_{ v\in B_{n-k}(X), uwv\in B_n(X)}\frac{g_{m+k+(n-k)}(\pi(x_{uwv}))}{\vert \pi^{-1}[\pi(uwv)]\vert} \\
\end{equation}
for each $x_{uv}\in [uv]$ and  $x_{uwv}\in [uwv]$.
Define for the fixed $u\in B_m(X)$ and $w_i\in W$, 
$$W_i=\sum_{uw_iv\in B_n(X), }\frac{g_{m+k+(n-k)}(\pi(x_{uw_iv}))}{\vert \pi^{-1}[\pi(uw_iv)]\vert}$$ for $i=1,\dots, N$, 
by choosing a point  $x_{uw_iv}$ from each cylinder set $[uw_iv]$. Since $g_n>0$ for each $n$, we obtain
\begin{equation*}
\begin{split}
\sum_{i=1}^{N}W_i
\leq M\sum_{\bar w\in B_{m+n}(X)}\frac{g_{m+n}(\pi(x_{\bar w}))}{\vert \pi^{-1}[\pi(\bar w)]\vert}\leq 
M\sum_{w\in B_{m+n}(Y)}\sup\{g_{m+n}(y_w): y_w\in [w]\}
\end{split}
\end{equation*}
for any choice of $x_{\bar{w}}\in [\bar w]$. 
Hence there exists $1\leq N_0\leq N$ such that 
 $$W_{N_0}\geq 
 \frac{M}{N}\sum_{w\in B_{m+n}(Y)}\sup\{g_{m+n}(y_w): y_w\in [w]\}.
 $$ 
 Since $X$ has the strong specification property, $Y$ also  has the property. Hence
given  $\tilde{v}\in B_{n-k}(Y)$ there exists $\tilde {w}\in B_{k}(Y)$ such that $\pi(u)\tilde{w}\tilde{v}\in B_{m+n}(Y)$. 
Let $\tilde{W}$ be the set consisting of all possible such $\tilde{w}$ and let
$\tilde{W}=\{\tilde{w}_1, \tilde{w}_2, \dots, \tilde{w}_S\}$ for some $S\in\N$. 
Define for the fixed $\pi(u)$ and $\tilde{w}_i \in \tilde{W}$, $$\tilde{W_i}=\sum_{\pi(u)\tilde{w}_i\tilde{v}\in B_{m+n}(Y), }g_{m+n}(y_{\pi(u)\tilde{w}_i\tilde{v}})$$
 for $i=1,\dots, S$, 
 by choosing a point $y_{\pi(u)\tilde{w}_i\tilde{v}}$ from each cylinder set $[\pi(u)\tilde{w}_i\tilde{v}]$. Then 
$$\sum_{i=1}^{S}\tilde{W}_i\leq M \sum_{\pi(u) v\in B_{n+m}(Y)}g_{m+n}(y_{\pi(u) v})\leq M^3 e^{C}g_m(y_{\pi(u)})
\sum_{w\in B_{n}(Y)}g_{n}(y_{w}) $$
for each $y_{\pi (u)v} \in [\pi (u)v], y_{\pi (u)} \in [\pi (u)], y_w\in [w]$, 
where $C$ is defined from the subadditive property (\ref{aa1}). 
Hence there exists $1\leq S_0\leq S$ such that

 \begin{equation}\label{eq2G}
 \tilde{W}_{S_0}\geq \frac{Me^C}{S}\sup\{g_m(y_{\pi(u)}): y_{\pi(u)}\in [\pi(u)]\}
\sum_{w\in B_{n}(Y)}\sup\{g_{n}(y_{w}):y_w\in [w]\} 
  \end{equation}
  for each $y_w\in [w]$. 
 Hence by  (\ref{eq2G})  
 \begin{equation*}
 \begin{split}
 &M\sum_{v\in F_n(u)}f_{m+n}(x_{uv})\geq W_{N_0}
\geq \frac{M}{N}\tilde W_{S_0} 
\\
&\geq \frac{M^2e^Cf_m(x_u)}{NS} \sum_{w\in B_{n}(X)}f_n(x_w)
 \geq \frac{f_m(x_u)}{l^{2k}} \sum_{w\in B_{n}(X)}f_n(x_w)
 \end{split} 
\end{equation*}
for each $x_u\in [u], x_{uv}\in [uv], x_w\in [w]$, for $n>k$.
Similarly, we can show the lower bound inequality of (\ref{LB}). 
\end{proof}
\begin{ex}
In Corollary \ref{ucase}, let $X$ be a topologically mixing shift of finite type and 
$g_n=({\phi_n})^{1/(1+\alpha)}$, $\alpha>0$. 
Then $\nu$ is invariant Gibbs for  $\F$ if and only 
if the upper bounds from equations (\ref{RB}) and  (\ref{LB}) hold, by Theorem \ref{general}.  
Equivalently, $\nu$ is Gibbs for  $\F$ if and only if (\ref {sup1}) and (\ref {sup2}) in Theorem \ref{mishiroU} hold.
\end{ex}
Now we study factors of invariant  Gibbs measures. The result generalizes \cite[Theorem 3.1]{yy} to arbitrary sequences of continuous functions on arbitrary subshifts.
\begin{prop}\label{yy}
Let $\pi: X \to Y$ be a one-block factor map between subshifts. If $X$ is balanced with respect to a sequence of continuous functions $\F=\{\log f_n\}_{n\in\N}$ on $X$,
then $Y$ is balanced with respect to another sequence of continuous functions $\G=\{\log g_n\}_{n\in\N}$.
If $\mu$ is an invariant Gibbs measure for $\F$, then
$$P(\F)=h_{\mu}(\sigma_X)+\lim_{n\to \infty}\frac{1}{n}\int \log f_nd\mu= 
h_{\pi\mu}(\sigma_Y)+\lim_{n\to \infty}\frac{1}{n}\int \log g_nd\pi\mu=P(\G).
$$
\end{prop}
\begin{proof}
We apply Theorem \ref{main} and the proof in \cite{yy}.  If $\mu$ is an invariant Gibbs measure for $\F$, 
there exists $M\geq 1$ such that
$ f_n(x)/f_n(y)\leq M$ for all $x, y\in [u], u\in B_n(X)$,  for all $n\in \N$.
For $y\in Y$, let $E_n(y)$ be a set consisting of exactly one point from each cylinder $[u]$, $u\in B_n(X)$
such that $ \pi([u])\subseteq [y_1\dots y_n].$
Then there exist $P\in \R$ and $C\geq 1$ such that 
\begin{equation*}
 \begin{split}
\frac{1}{e^{nP}MC} \sup_{E_n(y)} \{\sum_{x \in E_n(y)} f_n(x)\} \leq \sum_{\substack{u\in B_n(X)\\ \pi(u)=y_1\dots y_n}}\mu[u]
\leq \frac{C}{e^{nP}}
\sup_{E_n(y)} \{\sum_{x \in E_n(y)} f_n(x)\}
\end{split}
\end{equation*}
Define $g_n(y)=\sup_{E_n(y)} \{\sum_{x \in E_n(y)} f_n(x)\}$, $y\in Y$.  Then for each $n\in\N$, $g_n$ is locally constant. The last equalities hold by 
Theorem \ref{main} (see \cite[Theorem 4.5]{N} for the case when $f_n=e^{S_n(f\circ\pi)},  f\in C(Y)$ is H\"{o}lder).
\end{proof}

Acknowledgements: The author 
thanks Professor  De-Jung Feng  for his comment on the question regarding the Gibbs property of preimage measures. 
The author also thanks Dr. Reza Mohammadpour for  
useful discussion which helped to come up with  Example \ref{ex1} and motivated me to study factor maps.   



\begin{thebibliography}{99}
\bibitem{b2} L.\ Barreira.  {Nonadditive thermodynamic formalism: equilibrium and Gibbs measures.} {\em Discrete Contin. Dyn. Syst.} \textbf{16} (2006), 279--305.

      \bibitem{BG} S. Baker and A. E. Ghenciu. {Dynamical properties of $S$-gap shifts and other shift spaces}. 
      {\em J. Math. Anal. Appl. } \textbf{430} (2015), no. 2, 633-647.
 \bibitem{Bo} R.\ Bowen. 
 {Some systems with unique equilibrium states.} {\em Math. Syst. Theory} \textbf{8} (1974),
193--202.
    
     
  
 




\bibitem{CFH} Y. L.\ Cao, D.J.\ Feng and  W.\ Huang. 
{The thermodynamic formalism for sub-additive potentials.} {\em Discrete Contin. Dyn. Syst.}
 \textbf{20} (2008), 639--657.



\bibitem{Fe} D. J.\ Feng. {Equilibrium states for factor maps
between subshifts.} {\em Adv. Math.} \textbf{226} (2011), 2470--2502.






\bibitem{kim} M. \ Kim, 
{A necessary and sufficient condition for the existence of invariant Gibbs measures.} 
{\em Bull. Korean Math. Soc.} \textbf{61}  (2024), 1087–1105.

\bibitem{ly} C.\ Lacalle and Y. \ Yayama.  {On generalized compensation functions for factor maps between shift spaces on countable alphabets.} {\em Stoch. and Dyn. } \textbf{21} (2021), 2150012.
\bibitem{MQ} R.\ Mohammadpor and A. \ Quas. {Non-unique equilibrium measures and freezing phase transitions for matrix cocycles for negative t.} {\em Nonlinearity} \textbf{39}(2025), 015014. 
\bibitem{MV} R.\ Mohammadpor and P. \ Varandas. {Statistical properties of equilibrium states 
for fiber-bunched matrix cocyles and applications.} (2025) arXiv: 2508.057

\bibitem{m} A.\ Mummert.  {The thermodynamic formalism for almost-additive sequences.} 
{\em Discrete Contin. Dyn. Syst.} \textbf{16} (2006), 435-454.

\bibitem{N} T. Namiki. {The degree functions for cellular dynamics}. {\em Proc. Japan Acad.} 
\textbf{71}, Ser. A (1995), 10-12.
\bibitem{P} K. Park. { Quasi-multiplicativity of typical cocycles.} {\em Commun. Math. Phys.} 
\textbf{376}, (2020), 1957–2004.
\bibitem{R} T. Rush. 
{On the Superadditive Pressure for 1-Typical, One-Step,
Matrix-Cocycle Potentials.}
{\em Commun. Math. Phys. } (2024), 405-251.

\bibitem{w}P.\ Walters. {Relative pressure, relative equilibrium states, compensation functions
 and many-to-one codes between subshifts}.
{\em Trans. Amer. Math. Soc.} \textbf{296} (1986), 1-31.

\bibitem{w1} J. \ Wu. {Non-additive Thermodynamic Formalism and Applications.} 
Master’s Thesis, The Chinese University of Hong Kong, 2021.

\bibitem {Y2} Y.\ Yayama. { Existence of a measurable saturated compensation function between subshifts and
its applications}. { \em Ergod. Th. \&  Dynam. Sys.} \textbf{31} (2011), 1563--1589.


\bibitem {yy}Y. \ Yayama. {On factors of Gibbs measures for almost additive potentials}. 
{ \em Ergod. Th. \&  Dynam. Sys.} 
(2016), 276–309.




\bibitem{Yo} J.\ Yoo. 
{Decomposition of infinite-to-one factor cods and uniqueness of relative equilibrium states}.
{\em J. Mod.
Dyn.} \textbf{13} (2018), 271–284.

\end{thebibliography}
\end{document}